\documentclass{aptpub}

\newcommand{\al}{\alpha}                %%
\newcommand{\ga}{\gamma}                %% abbreviated
\newcommand{\ra}{\rightarrow}           %%
\authornames{Yingdong Lu} % insert the authors here for use in running head. If three or more authors please use (for example) M.~YARROW {\it et al}. Author names should follow the same M.~YARROW format and if two authors, separate by 'AND'.
\shorttitle{Stein Method for Curie-Weiss Model} % insert short title here for use in running head

\begin{document}%\recd{}{}%Do not alter this line.

\title{Approximating the Markov Chain of the Curie-Weiss Model} % insert title - use \\ if it requires more than one line.

\authorone[IBM Research]{Yingdong Lu} 
%Affiliation is just the name of your university or institution, for example 'University of Sheffield'. Author names should be of the form 'Mark Yarrow'. 
%Authors should be ordered alphabetically subject to the convention in that particular authors country. For example 'Remco van der Hofstad' would be listed under 'H' as is standard in the Netherlands. 

%Please use the following format for addresses and emails. The APT office will sort this out after you submit your files.
\addressone{1101 Kitchawan Rd, Yorktown Heights, NY 10598} % Your postal address goes here.
\emailone{yingdong@us.ibm.com} %Authors email goes here.

\begin{abstract}
In this paper, we quantify some known approximation to the Curie-Weiss model via applying the Stein method to the Markov chain whose stationary distribution coincides with Curie-Weiss model.% text of abstract goes here!
\end{abstract}

\keywords{}%insert keywords separated by a semicolon. You should avoid including keywords which also appear in the title.

\ams{}{} % insert the primary Maths Subject Classification number in the first bracket
         % and the secondary ams number(s) in the second bracket
         % e.g. \ams{60E20}{49G03;49F10}
\section{Introduction}
\label{sec:intro}

For an integer $n>0$ and real numbers $\beta\ge 0$ and $h$, the Curie-Weiss model for $n$ spins at temperature $1/\beta$ refers to a probability measure on $S^n:=\{-1, 1\}^n$, given by,
\begin{align*}
\pi(x) = Z_n \exp [ -\beta H^n(x)], \quad \forall x \in S^n, 
\end{align*}
with
\begin{align*}
H^n(x) = \frac{1}{2n} \sum_{i,j=1}^n x_i x_j - h \sum_{i=1}^n x_i, \quad \forall x \in S^n, 
\end{align*}
and $Z_n$ are the normalization coefficients. The parameter $h$  relates to external magnetization. This is an important mathematical model for studying the interaction of electron spins in real ferromagnets, and is sometimes also called the Ising model on the complete graph. For detailed physics background and thorough analysis on the Curie-Weiss model, see, e.g.~\cite{ellis2006entropy}. The system has two phases determined by the temperature. $0\le \beta <1$ is known as the {\it supercritical phase}, and $\beta=1$ is known as the {\it critical phase}. While the final results for these two phases are different, there are commonalities in analysis, and we will discuss them separately only when necessary.  

Define the one-dimensional quantity $m^n(x) :=\frac{1}{n} \sum_{i=1}^n x_i$, that is also known as the magnetization of the system. Thus, $H^n(x)$ can be rewritten in the following form,
\begin{align*}
H^n(x) = -n \left[ \frac{1}{2} (m^n(x))^2 + h m^n(x)\right].
\end{align*}
The concentration of mass happens at the minimizer of the following function, 
\begin{align*}
i(m) = -\left(\frac12 \beta (m)^2 + \beta h m \right) + \frac{1-m}{2} \log (1-m) + \frac{1+m}{2} \log (1+m).
\end{align*}
It is given by,
\begin{align}
\label{eqn:basic_relation} 
\beta m_0 +\beta h = \frac12 \log \frac{1+m_0}{1-m_0},
\end{align}
or equivalently, $m_0 =\tanh (\beta (m_0+ h))$.

In certain parameter cases, it is observed that, see e.g. ~\cite{bierkens2017}, under proper scaling, the Markov chain that produces the density can converge to diffusion processes. Hence, the stationary distribution of the diffusion process can be viewed as a good approximation of the Curie-Weiss model. Note that the key to this approximation is the concentration effect proved, see, e.g.~\cite{chatterjee2010, chatterjee2011}. They have been proved for a special case of critical system and for supercritical cases. For other cases, the concentration is known to be more difficult, and will be part of our future research.

In this paper, we aim to quantify the discrepancy in their stationary distributions via applying Stein method. In the  analysis, we will rely on a detailed analysis of the infinitesimal operators of the two processes. While process level approximation can be achieved by a semigroup expansion method as shown in cite{princeton}. More detailed regularity analysis for the solution to the Poisson equation related to the generator is required for the Stein method. More specifically, our approach is developing a Stein method approximation through a Markov chain generated by the Metropolis-Hasting method in simulation. The stationary distribution of this Markov chain is $\pi_n$. Meanwhile, $\pi_\infty$ will be the stationary distribution of a diffusion process $Y(t)$, and is also denoted by $Y(\infty)$.  To estimate the difference between $\pi_n$ and $Y(\infty)$ in the weak sense, we need to quantify the values of $|E[h(\pi_n)]-E[h(Y_\infty)]|$ for an arbitrary bounded function. The key of the Stein method is to reduce the estimation of this quantity to that of  $|E[G(f_h(Y_\infty)]-E[G_n(f_h(\pi_n)]|$, with $f_h$ being the solution to the Stain equation with respect to $h$, with the hope that the structural properties of solution will reveal more information that can aid the calculation of the quantities. 

While better result on the distribution function were obtained in Chaterjee using techniques that can not be easily generalized, the results in this paper is focus on general function evaluated at the stationary distribution. Meanwhile, similar method was used by Braverman and Dai~\cite{braverman2017}and Gurvich~\cite{gurvich2014,doi:10.1287/moor.2013.0593} for stocahstic processing system and networks. 

In the rest of paper, we will first introduce the Markov chain in Sec. \ref{sec:MC}; iwe will then discuss the Stein method and present the results on approximating the stationary distributions in Sec. \ref{sec:stein}. 

\section{The Markov Chain}
\label{sec:MC}

In this section, we present the Markov chain whose stationary distribution is the Curie-Weiss model, then write out the scaled and centered version and introduce some of its basic properties. 

\subsection{The Definition of the Markov Chain} 

First, let $X^n_t$ denote a Markov chain whose state space is $S^n$, and its transition probability is given by, 
\begin{align}
\label{eqn:MC-Transition-orig}
P_{x,y} = \left\{ \begin{array}{cc} \frac{1}{n} & y \in N_0(x) \\ 0 & y \notin N_0(x) \end{array} \right. 
\end{align}
where $N_0(x) =\cup_{k=1}^n \{ y \in S^n , y \neq x , y-x = \pm 2e^k \}$. The Markov chain will flip one of its coordinates with probability $1/n$.  

%Define a map $\eta: S^n \rightarrow {\mathbb R}$, $\eta(x) := n^\gamma (m^n(x) -m_0)$ with $m^n(x): = (1/n) \sum_{k=1}^n x_k$ and $m_0$ as given before. 

%$Z_t^n=\eta(X_t)$ is a Markov chain that is governed by 
%\begin{align}
%\label{eqn:MC-Transition-scaled}
%Q_{\eta, \eta\pm 2n^{\gamma -1} } = \frac12 (1\mp (m_0 + n^{-r} \eta)).
%\end{align}
%for all $\eta \in [n^{\gamma} (-1-m_0), n^{\gamma} (1-m_0)]$.

Let $Y_t^n=\eta^n(X_t)$, where,
\begin{align*}
m^n(x): = (1/n) \sum_{k=1}^n x_k, \quad 
\eta^n (x) := n^{\gamma} (m^n(x) -m_0),
\end{align*}
which is the centralized and renormalized magnetization. Its (one-step) transition rate is characterized by,
\begin{align*}
Q^n(\eta^n, \eta^n+2n^{\gamma-1}) &= \frac12[1-(m_0+ n^{-\gamma} \eta^n)],\\ Q^n(\eta^n, \eta^n-2n^{\gamma-1}) &= \frac12[1+(m_0+ n^{-\gamma} \eta^n)].
\end{align*}
Finally, $Z_t^n$, the Metropolis-Hasting version of this Markov chain, is the one with the following transition rates, 
\begin{align*}
P^n( \eta^n, \eta^n \pm 2n^{\gamma-1}) = Q^n(\eta^n, \eta^n \pm 2n^{\gamma-1})(1\wedge \exp [\beta (\Phi^n(\eta) - \Phi^n(\eta \mp 2n^{\gamma -1}))],
\end{align*}
with
\begin{align*}
\Phi^n(\eta) = -\frac12 n^{1-2\gamma} \eta^2 -n^{1-\gamma} (m_0+h)\eta.
\end{align*}
Define, 
\begin{align*}
p_{\pm}^n (\eta) : &= P^n (\eta, \eta\pm 2n^{\gamma -1})\\
& = \left[\frac12   [1\mp (m_0 + \frac{\eta}{n\delta})\right]\left\{1\wedge \exp\left[\pm \beta \left(\frac{\eta}{n \delta } \pm 2m_0\pm 2h + \frac{2}{n} \right)\right] \right\}.
\end{align*}
Furthermore, there will be a time scaling, or speed up of the Markov chain, that is quantified by the number $\alpha$. The relationship under which a meaningful process limit can be established has been identified in~\cite{bierkens2017}. More specifically $\al =1-\ga$.

\subsection{Estimating Some Basic Quantities of the Markov Chain}

The concentration of the Curie Weiss, a known result from Chatterjee, determined that $\gamma=\frac12$ for $0\le \beta <1$ and $\gamma =\frac{1}{4}$ for $\beta=1$. More specifically, we know that,

\noindent
{\bf Lemma}[Chatterjee]
\begin{itemize}
	\item[1.] For all $\beta \ge 0$, $h \in {\mathbb R}$, we have, for any $t\ge 0$,
	\begin{align}
	\label{eqn:tail_one} 
	\pi^n\left( |m^n -\tanh (\beta (m^n+h))| \ge \frac{\beta}{n}+\frac{t}{\sqrt{n}}\right)\le 2 \exp \left(-\frac{t^2}{4(1+\beta)}\right).
	\end{align}
	\item[2.] For $h=0$ and $\beta=1$, we have, 
	\begin{align}
	\label{eqn:tail_two} 
	\pi^n(|m^n|\ge t^{1/4} ) \le 2e^{-cnt}
	\end{align}
\end{itemize}

We can use the tail distribution estimation to guarantee the boundedness of the moments, which will be useful later.

\begin{lem}
	\label{lem:finite_moments}
	The finiteness of the moments. More specifically, there exists an constant $C_m>0$, such that,
	\begin{align}
	\label{eqn:moments}
	E[Z^n] \le C_m^n n!!,
	\end{align}
where $n!!$ denotes the {\it double factorial} of $n$, i.e. the product of all the integers from $1$ to $n$ that have the same parity as $n$. 
\end{lem}
\begin{proof}
	We will discuss the {\it critical} and {\it supercritical} cases separately. 
	
	For the critical case, $\eta> n^{\delta}$, we know that, 
	\begin{align*}
	\pi[|\eta|\ge n^{\delta}]=\pi_n[|m^n|\ge n^{\delta-\frac{1}{4}}] \le \exp (-cn^{1+\delta-\frac{1}{4}}).
	\end{align*}
	%we can have,
	%\begin{align*}
	%\ex\left[\frac{\eta^4}{n^{\frac{1}{4}}}\right] \le \frac{1}{n^{\frac{1}{4}}} \left\{\ex\left[\eta^4{\bf 1} \{ \eta \le n^{\delta}\}\right]+\ex\left[\eta^4{\bf 1} \{ \eta > n^{\delta}\}\right]\right\}
	%\end{align*}
	%The key is the second term. 
	For any integer $k$, we have,
	\begin{align*}
	\pi[|\eta|\ge k]=\pi_n\left[|m^n|\ge \frac{k}{n^{1/4}} \right] = \pi_n\left[|m^n|\ge \left(\frac{k^4}{n}\right)^{1/4} \right]\le e^{-ck^4},
	\end{align*}
	The inequality comes from inequality \eqref{eqn:tail_two} with $t= \frac{k^4}{n}$. This will ensure the finiteness of the moments. 
	
	For the supercritical case, again, for any integer $k$, we have,
	\begin{align*}
	\pi[|\eta|\ge k]&= \pi^n(|m^n-m_0| \ge \frac{k}{\sqrt{n}})\\ &=\pi^n\left(|m^n-m_0| \ge \frac{\beta}{n} + \frac{k-\frac{\beta}{\sqrt{n}}}{\sqrt{n}}\right)\\ &\le
	2\exp\left(-\frac{(k-\frac{\beta}{\sqrt{n}})^2}{4(1+\beta)}\right).
	\end{align*}
	This implies the finiteness of the moment estimate. 
\end{proof}
\begin{rem}
The bound is not necessary optimal, just suffice for our purpose in this paper. 
\end{rem}

The dynamics of the Markov chain discussed above indicates that the quantities $[p^n_+ (\eta) +p^n_-(\eta)] $ and $[p^n_+ (\eta) -p^n_-(\eta)] $ play prominent roles in the quantification of the process. In the following, we provide a detailed analysis on these two quantities. 
\begin{lem}
	\label{lem:the_plus_calculation}
	The calculation of 
	\begin{align}
	\nonumber [p^n_+ (\eta) +p^n_-(\eta)] = &\frac12  [ 1- m_0] \left\{1+\exp\left[\beta \left(2\frac{\eta}{n \delta } + \frac{2}{n} \right)\right]\right\} \\ & -  \frac12   \frac{\eta}{n\delta} \left\{1- \frac{1-m_0}{1+m_0}\exp\left[\beta \left(2\frac{\eta}{n \delta }  + \frac{2}{n} \right)\right] \right\}
	\label{eqn:the_plus_calculation}
	\end{align}
\end{lem}
\begin{proof}
	The calculations goes as the following, 
	\begin{align*}
	& [p^n_+ (\eta) +p^n_-(\eta)] 
	\\ =&\frac12  \left[ 1- \left(m_0 + \frac{\eta}{n\delta}\right)\right] + \frac12  \left[ 1+ \left(m_0 + \frac{\eta}{n\delta}\right)\right] \exp\left[\beta \left(2\frac{\eta}{n \delta } - 2m_0- 2h + \frac{2}{n} \right)\right] 
	\\ =& \frac12  [ 1- m_0] + \frac12  [ 1+ m_0] \exp[-2\beta(m_0+h)]\exp\left[\beta \left(2\frac{\eta}{n \delta } + \frac{2}{n} \right)\right] \\ & -  \frac12   \frac{\eta}{n\delta} \left\{1- \exp\left[\beta \left(2\frac{\eta}{n \delta } - 2m_0- 2h + \frac{2}{n} \right)\right] \right\}
	\\=& \frac12  [ 1- m_0] \left\{1+\exp\left[\beta \left(2\frac{\eta}{n \delta } + \frac{2}{n} \right)\right]\right\} -\frac12   \frac{\eta}{n\delta} \left\{1- \exp\left[\beta \left(2\frac{\eta}{n \delta } - 2m_0- 2h + \frac{2}{n} \right)\right] \right\}
	\\=& \frac12  [ 1- m_0] \left\{1+\exp\left[\beta \left(2\frac{\eta}{n \delta } + \frac{2}{n} \right)\right]\right\}  -  \frac12   \frac{\eta}{n\delta} \left\{1- \frac{1-m_0}{1+m_0}\exp\left[\beta \left(2\frac{\eta}{n \delta }  + \frac{2}{n} \right)\right] \right\}.
	\end{align*}
	The first equation is just simple algebraic manipulation, and the second and third equations used the basic relationship in \eqref{eqn:basic_relation}.
	% 
	%\begin{align*} 
	%\beta m_0 +\beta h = \frac12 \log \frac{1+m_0}{1-m_0},
	%\end{align*}
\end{proof}
Expand the exponential function gives us the following order estimation.
\begin{cor}
	\label{cor:the_plus_calculation} We have the following evaluations on terms of the expression in \eqref{eqn:the_plus_calculation},
	\begin{itemize}
		\item[i]
		The $0$-th order term in the expression: $(1-m_0)$.
		\item[ii]
		The $(n\delta)^{-1}$ term, 	$E[\eta] \left[ \beta(1-m_0) -\frac{m_0}{1+m_0}\right].$
		\begin{itemize}
			\item
			Recall that, in the critical case, we have, $\delta = n^{-3/4}$, so $(n\delta)^{-1}$ terms is $n^{-1/4}$,
			\item
			In the supercritical system,  $\delta = n^{-1/2}$, so $(n\delta)^{-1}$ terms is $n^{-1/2}$.
		\end{itemize}
	\item[iii]
	Everything else will be higher order. 
	\end{itemize}
In other words, we have, $p^n_+ (\eta) +p^n_-(\eta) =(1-m_0)+E_1(\eta)$, with $E_1(\eta)= \eta \left[ \beta(1-m_0) -\frac{m_0}{1+m_0}\right](n\delta)^{-1} + O((n\delta)^{-2})$. 
\end{cor}
\begin{proof}
	Both (i) and (ii) are straightforward, and (iii) follows from Lemma \ref{lem:finite_moments}.
\end{proof}

\begin{lem}
	\label{lem:the_minus_calculation}
	The calculation of $[p^n_+ (\eta) -p^n_-(\eta)] $. 
	\begin{align}
	\nonumber p^n_+ (\eta) -p^n_-(\eta)=& \frac12  (1- m_0) \left\{1-\exp\left[\beta \left(2\frac{\eta}{n \delta } + \frac{2}{n} \right)\right]\right\} \\ & -  \frac12   \frac{\eta}{n\delta} \left\{1+\frac{1-m_0}{1+m_0} \exp\left[\beta \left(2\frac{\eta}{n \delta } + \frac{2}{n} \right)\right] \right\} \label{eqn:the_minus_calculation}
	\end{align}
\end{lem}
\begin{proof}
	Recall
	\begin{align*}
	p_{\pm}^n (\eta) 
	& = \frac12  \left[ 1\mp \left(m_0 + \frac{\eta}{n\delta}\right)\right]\left\{1\wedge \exp\left[ \beta \left(2\frac{\eta}{n \delta } \pm 2m_0\pm 2h + \frac{2}{n} \right)\right] \right\}.
	\end{align*}
	When $2m_0+ 2h>0$, then, for large enough $n$,  $2\frac{\eta}{n \delta } + 2m_0+ 2h + \frac{2}{n} \ge 0$, and  $2\frac{\eta}{n \delta } - 2m_0 - 2h + \frac{2}{n} \le 0$, we know that, 
	\begin{align*}
	& [p^n_+ (\eta) - p^n_-(\eta)] \\ =&\frac12  \left[ 1- \left(m_0 + \frac{\eta}{n\delta}\right)\right] -  \frac12  \left[ 1+ \left(m_0 + \frac{\eta}{n\delta}\right)\right] \exp\left[\beta \left(2\frac{\eta}{n \delta } - 2m_0- 2h + \frac{2}{n} \right)\right] 
	%\\ =& \frac12 \left\{ 1-  \exp\left[\beta \left(2\frac{\eta}{n \delta } - 2m_0- 2h + \frac{2}{n} \right)\right] \right\} \\ &- \left(m_0 + \frac{\eta}{n\delta}\right)\left\{ 1+  \exp\left[\beta \left(2\frac{\eta}{n \delta } - 2m_0- 2h + \frac{2}{n} \right)\right] \right\}.
	\end{align*}
	Meanwhile, we know that $$\exp[\beta(m_0+h)] = \left(\frac{1+m_0}{1-m_0}\right)^{\frac12}.$$ 
	So,
	\begin{align*}
	& [p^n_+ (\eta) - p^n_-(\eta)] \\ =&\frac12  \left[ 1- \left(m_0 + \frac{\eta}{n\delta}\right)\right] -  \frac12  \left[ 1+ \left(m_0 + \frac{\eta}{n\delta}\right)\right] \exp\left[\beta \left(2\frac{\eta}{n \delta } - 2m_0- 2h + \frac{2}{n} \right)\right] 
	\\ =& \frac12  (1- m_0) - \frac12  (1+ m_0) \exp[-2\beta(m_0+h)]\exp\left[\beta \left(2\frac{\eta}{n \delta } + \frac{2}{n} \right)\right] \\ & -  \frac12   \frac{\eta}{n\delta} \left\{1+ \exp\left[\beta \left(2\frac{\eta}{n \delta } - 2m_0- 2h + \frac{2}{n} \right)\right] \right\}\\=& \frac12  (1- m_0) \left\{1-\exp\left[\beta \left(2\frac{\eta}{n \delta } + \frac{2}{n} \right)\right]\right\}  -  \frac12   \frac{\eta}{n\delta} \left\{1+ \exp\left[\beta \left(2\frac{\eta}{n \delta } - 2m_0- 2h + \frac{2}{n} \right)\right] \right\} \\=& \frac12  (1- m_0) \left\{1-\exp\left[\beta \left(2\frac{\eta}{n \delta } + \frac{2}{n} \right)\right]\right\} -  \frac12   \frac{\eta}{n\delta} \left\{1+\frac{1-m_0}{1+m_0} \exp\left[\beta \left(2\frac{\eta}{n \delta } + \frac{2}{n} \right)\right] \right\} 
	\end{align*}
	The case of $2m_0+ 2h<0$ can be similarly considered. 
\end{proof}
Similar to Corollary \ref{cor:the_plus_calculation}, we have, 
\begin{cor}
	\label{cor:the_minus_calculation}
	The terms of \eqref{eqn:the_minus_calculation} can be evaluated as,
	\begin{itemize}
		\item 
		In the case of $h=0$ and $\beta=1$, we have, $m_0=0$, with $\al=3/4$ and $\gamma =1/4$, 
		$$2n^{\frac{3}{4}} [p^n_+ (\eta) - p^n_-(\eta)]=\left[ -\frac{2}{3} n^{-\frac{3}{4}}\eta^3 + \frac{4}{3}n^{-1} \eta^4 \right]n^{\frac{3}{4}}.$$
		\item
		In the supercritical case
		\begin{align*}
		&-( 1- m_0)\eta \left[\beta  +  \frac{1}{2(1+m_0)}\right]\frac{1}{n^\frac12} - ( 1- m_0)\beta\left[(1 +\eta^2 \beta) -\frac{ \eta^2}{1+m_0}\right] \frac{1}{n} 
		\\ &- (1- m_0)\beta \eta \left[2\beta     -\frac{ (1 +\eta^2 \beta)}{1+m_0} \right] \frac{1}{n^{\frac{3}{2}}} \end{align*}
	\end{itemize}
In other words, we have, $n^{\frac12}[p^n_- (\eta) +p^n_-(\eta)] =-( 1- m_0)\eta \left[\beta  +  \frac{1}{2(1+m_0)}\right]+E_2(\eta)$, with $E_2(\eta)=O(\frac{1}{n^\frac12})$.
\end{cor}

\begin{proof}	
	For the critical case in which $h=0$ and $\beta=1$, we al so have, $m_0=0$, with $\al=3/4$ and $\gamma =1/4$, thus, the following estimation follows directly from the expansion of the exponential function,
	\begin{align*}
	&2n^{\frac{3}{4}} [p^n_+ (\eta) - p^n_-(\eta)] \\ =& \left[(1-n^{-\frac{1}{4}}\eta) - (1+ n^{-\frac{1}{4}}\eta)(1-2n^{-\frac{1}{4}}\eta 
	+ 2n^{-\frac12}\eta^2 - \frac{4}{3} n^{-\frac{3}{4}}\eta^3)\right] n^{\frac{3}{4}}+O(n^{-\frac12})\\= & \left[ -\frac{2}{3} n^{-\frac{3}{4}}\eta^3 + \frac{4}{3}n^{-1} \eta^4 \right]n^{\frac{3}{4}} + O(n^{-\frac12}).
	\end{align*}
	The first term is, of course, the desired drift term in the generator, so we only need to estimate the second term, i.e. $E\left[\frac{4 \eta^4}{3n^{-1/4}}\right]$. The estimation of the tail order again follows from Lemma 
	\ref{lem:finite_moments}.
	
	Now, let us consider the supercritical case. 
	With $\gamma=\frac12$, we have, 
	\begin{align*}
	1-\exp\left[\beta \left(2\frac{\eta}{n^\frac12} + \frac{2}{n} \right)\right] & = - \beta \left(2\frac{\eta}{n^\frac12} + \frac{2}{n} \right) -\frac{1}{2}\beta^2 \left(2\frac{\eta}{n^\frac12} + \frac{2}{n} \right)^2 + O\left(\frac{1}{n^{\frac{3}{2}}}\right) \\ &= - 2 \beta \frac{\eta}{n^\frac12} -(2\beta +2\eta^2 \beta^2) \frac{1}{n}+ O\left(\frac{1}{n^{\frac{3}{2}}}\right) %- 4\beta^2\frac{\eta}{n^{\frac{3}{2}}}- \frac{ 2\beta^2}{n^2}
	\end{align*}
	Plug it in \eqref{eqn:the_minus_calculation}, we get,
	\begin{align*}
	%&\frac12  [ 1- m_0] \left\{- 2 \beta \frac{\eta}{n^\frac12} -(2\beta +2\eta^2 \beta^2) \frac{1}{n} - 4\beta^2\frac{\eta}{n^{\frac{3}{2}}}- \frac{ 2\beta^2}{n^2}
	%\right\} \\ & -  \frac12   \frac{\eta}{n^\frac12} \exp\left[\beta \left( - 2m_0- 2h \right)\right]\left[1+ 2 \beta \frac{\eta}{n^\frac12} +(2\beta +2\eta^2 \beta^2) \frac{1}{n} +4\beta^2\frac{\eta}{n^{\frac{3}{2}}}+\frac{ 2\beta^2}{n^2}\right] \\ =& \left\{-\beta \eta  [ 1- m_0]  -  \frac12\eta  \exp[\beta \left( - 2m_0- 2h \right)]  \right\}\frac{1}{n^\frac12} \\ & -[[ 1- m_0] (\beta +\eta^2 \beta^2)+\beta\eta^2 \exp[\beta \left( - 2m_0- 2h \right)] ]  \frac{1}{n} \\ & -\{2\beta^2  \eta [ 1- m_0]  -\eta  \exp[\beta \left( - 2m_0- 2h \right)] (\beta +\eta^2 \beta^2)  \} \frac{1}{n^{\frac{3}{2}}} 
	%\\ =&   -( 1- m_0) \left[\beta \eta  +  \frac{\eta}{2(1+m_0)}\right]\frac{1}{n^\frac12} - ( 1- m_0)\left[(\beta +\eta^2 \beta^2) -\frac{\beta \eta^2}{1+m_0}\right] \frac{1}{n} 
	%\\ &- (1- m_0)\left[2\beta^2  \eta   -\frac{\eta   (\beta +\eta^2 \beta^2)}{1+m_0} \right] \frac{1}{n^{\frac{3}{2}}} 
	 -( 1- m_0)\eta \left[\beta  +  \frac{1}{2(1+m_0)}\right]\frac{1}{n^\frac12} - ( 1- m_0)\beta\left[(1 +\eta^2 \beta) -\frac{ \eta^2}{1+m_0}\right] \frac{1}{n} + O\left(\frac{1}{n^{\frac{3}{2}}}\right)
	%\\ &- (1- m_0)\beta \eta \left[2\beta     -\frac{ (1 +\eta^2 \beta)}{1+m_0} \right] \frac{1}{n^{\frac{3}{2}}} 
	\end{align*}
\end{proof}

\section{Approximation via the Stein Method}
\label{sec:stein}

In this section, we will present the result on the approximation of the stationary distribution via the Stein method. First, we will provide some quantitative characterization of the solution to the Stein equation in Sec. \ref{sec:solution}; then we will present and demonstrate the main result in Sec. \ref{sec:main}. 

\subsection{Solution to the Stein equations}
\label{sec:solution}

To study the solution to the Stein equation \eqref{eqn:stein}, we can examine the following equation that is in a more general form. For functions $a(x)$ and $b(x)$ satisfying that $a(x)$ and 
$b(x)/a(x)$ are absolutely continuous, and $\frac{e^{\int_0^y \frac{2b(u)}{a(u)} du }}{a(y)}$ is integrable, consider the equation, 
\begin{align}
\label{eqn:ode_general}
\frac12 a(x) f''_h(x) + b(x) f'_h(x) = E h(Y)-h(x), 
\end{align}
with  $\lim_{x\ra -\infty} f(x) =0 $ and $\lim_{x\ra -\infty} f'(x) =0$ where $C_Y\frac{e^{\int_0^y \frac{2b(u)}{a(u)} du }}{a(y)}$ represents the density of random variable $Y$ with $C_Y$ being the normalizing coefficient. 

From basic differential equation  calculation, we know that the solution to \eqref{eqn:ode_general} can be written as 
\begin{align*}
f(x) =\int_0^x \int_{-\infty}^y e^{\frac{\int_0^z \frac{2b(u)}{a(u)} du -\int_0^y \frac{2b(u)}{a(u)} du }{4}} \frac{2  [Eh(Y)-h(z)] }{a(z)}dzdy.
\end{align*}
Hence,
\begin{align}
\label{eqn:1st_derivative_stein_soln}
f'_h(x) &= e^{-\int_0^x \frac{2b(u)}{a(u)} du } \int_{-\infty}^x e^{\int_0^y \frac{2b(u)}{a(u)} du }  \frac{2  [Eh(Y)-h(y)] }{a(y)} dy,
\end{align}
\begin{align}
\label{eqn:2nd_derivative_stein_soln}
f''_h(x) = -\frac{2b(x)}{a(x)} f'_h(x) + \frac{2  [Eh(Y)-h(x)] }{a(x)},
\end{align}
\begin{align}
\label{eqn:3rd_derivative_stein_soln}
f'''_h(x)= - \left(\frac{2b(x)}{a(x)} \right)'f'_h(x) -\frac{2b(x)}{a(x)} f''_h(x) - \frac{2  [h'(x)] }{a(x)} - \frac{2 a'(x) [Eh(Y)-h(x)] }{a^2(x)}.
\end{align}
Reexamine the first derivative, $f'_h$, we find, 
\begin{align*}
f'_h(x) & = e^{-\int_0^x \frac{2b(u)}{a(u)} du } \int_{-\infty}^x e^{\int_0^y \frac{2b(u)}{a(u)} du }  \frac{2  [Eh(Y)-h(y)] }{a(y)} dy \\ 
& = C_Ye^{-\int_0^x \frac{2b(u)}{a(u)} du } \int_{-\infty}^x  \int_{-\infty}^\infty e^{\int_0^y \frac{2b(u)}{a(u)}du + \int_0^z \frac{2b(u)}{a(u)}  dy}\frac{2  [h(z)-h(y)] }{a(z)a(y)}dz dy
\end{align*}
For $x\ge 0$, from integration by part, we have, 
\begin{align*}
%f'(x)   & = e^{-\int_0^x \frac{2b(u)}{a(u)} du } \int_{-\infty}^x  \int_{-\infty}^\infty e^{\int_0^y \frac{2b(u)}{a(u)}du + \int_0^z \frac{2b(u)}{a(u)}  dy}\frac{2  [h(z)-h(y)] }{a(z)a(y)}dz dy
%\\ & = e^{-\int_0^x \frac{2b(u)}{a(u)} du }  \int_{-\infty}^\infty \frac{e^{\int_0^z \frac{2b(u)}{a(u)} du }}{a(z)}\left[\int_{-\infty}^x e^{\int_0^y \frac{2b(u)}{a(u)}du }\frac{2  [h(z)-h(y)] }{a(y)}dy \right]dz
%\\ & =  2e^{-\int_0^x \frac{2b(u)}{a(u)} du }  \int_{-\infty}^\infty\left[\frac{e^{\int_0^z \frac{2b(u)}{a(u)} du }}{a(z)}h(z) F(x)- \frac{e^{\int_0^z \frac{2b(u)}{a(u)} du }}{a(z)}F^{(h)}(x) \right]dz
%\\ & =  2e^{-\int_0^x \frac{2b(u)}{a(u)} du } \left[Eh(Y) F(x)- F^{(h)}(x) \right]
%\\ & = 2e^{-\int_0^x \frac{2b(u)}{a(u)} du }  \left[Eh(Y) %F(x)- (Eh(Y)- {\bar F}^{(h)}(x) )\right]
%\\
f'(x)&=  2e^{-\int_0^x \frac{2b(u)}{a(u)} du } \left[{\bar F}^{(h)}(x)-Eh(Y){\bar F}(x)\right],
\end{align*}
where 
\begin{align*}
F(x) &= \int_{-\infty}^x\frac{e^{\int_0^z \frac{2b(u)}{a(u)} du }}{a(z)} dz, & F^{(h)}(x) &= \int_{-\infty}^x\frac{e^{\int_0^z \frac{2b(u)}{a(u)} du }}{a(z)} h(z) dz, \\  {\bar F} (x) &=1- F(x), & {\bar F}^{(h)} (x) &=E h(Y)- F^{(h)}(x).
\end{align*}

Since, 
\begin{align*}
\left[\frac{e^{\int_0^x \frac{2b(u)}{a(u)} du }}{a(x)}\right]' & = \frac{e^{\int_0^x \frac{2b(u)}{a(u)} du }}{a(x)} \left[\frac{b(x) -a'(x)}{a(x)}\right], 
\end{align*}
Again by integration by part, we have, for $x\ge 0$, 
\begin{align*}
%{\bar F} (x) &= \int_{x}^\infty\frac{e^{\int_0^z \frac{2b(u)}{a(u)} du }}{a(z)} dz 
%\\ &=  \int_{x}^\infty\frac{e^{\int_0^z \frac{2b(u)}{a(u)} du }}{a(z)} \left[\frac{b(z) -a'(z)}{a(z)}\right]  \left[\frac{b(z) -a'(z)}{a(z)}\right]^{-1} d z
%\\ &\le  \left[\frac{b(x) -a'(x)}{a(x)}\right]^{-1} \int_{x}^\infty d \left[\frac{e^{\int_0^z \frac{2b(u)}{a(u)} du }}{a(z)}\right]
%\\
{\bar F} (x)& =  -\left[\frac{b(x) -a'(x)}{a(x)}\right]^{-1}\frac{e^{\int_0^x \frac{2b(u)}{a(u)} du }}{a(x)}.
\end{align*}
The inequality is due to the monotonicity assumption of 
\begin{align*}
\left[\frac{b(x) -a'(x)}{a(x)}\right]^{-1}
\end{align*}
\begin{lem}
	\label{lem: gradient_bound}
	Under the condition that $\frac{b(x) -a'(x)}{a(x)}$ is strictly positive and increasing, we have, 
	\begin{align*}
	f'(x)   & \le -\left[\frac{b(x) -a'(x)}{a(x)}\right]^{-1}(\|h\|_1+ \|h\|_\infty),
	f''(x) & \le \frac{2b(x)}{a(x)}  \left[\frac{b(x) -a'(x)}{a(x)}\right]^{-1}(\|h\|_1+ \|h\|_\infty).
	\end{align*}
\end{lem}
From the parameters we have, as well as \eqref{eqn:2nd_derivative_stein_soln}, we can certainly conclude that 
\begin{cor}
	\label{cor:boundedness_of_the_third_derivative}
	There exist a $C_3>0$, such that $|f'''(x)|\le C_3$, for all $x\in {\mathbb R}$. 
\end{cor}
\begin{rem}
	Our problem certainly satisfies the condition on the density. Similar, we can also have the bound on the second derivative. 
\end{rem}

\subsection{Approximating the Stationary Distributions}
\label{sec:main}

For the Curie-Weiss models, this method will lead the following result. 
\begin{thm}
	\label{thm:main}
	For proper parameter range, we have the following estimation, especially,
	\begin{itemize}
		\item For the critical case, i.e. $\beta=1$,
		\begin{align}
		\label{eqn:main}
		E [h(X(\infty))] - E[h(Y(\infty))] \le \frac{C}{n^{\frac{1}{4}}}.
		\end{align}
		\item
		For the super-critical case,
		\begin{align}
		\label{eqn:main_alt}
		E [h(X(\infty))] - E[h(Y(\infty))] \le \frac{C}{n^{\frac{1}{2}}}.
		\end{align}
	\end{itemize}
\end{thm}

% This the proof of the main theorem, with the technical stuff in the appendix
\begin{proof}
Let us denote the generator for the $n$-th system as $G_n$, and the limit as $G_\infty$. By the above arguments, what we need to estimate is,
\begin{align*}
E[ G_n f( X(\infty)- G_\infty f(X(\infty))]= E[ (G_n f- G_\infty f)( X(\infty)].
\end{align*} 
For this purpose, write,
\begin{align}
	G_n f(\eta)= & n^\al [P^n f(\eta) -f(\eta)] \nonumber \\
	= & n^\al \left[\frac12   [1-(m_0 + \frac{\eta}{n\delta})\right] \left\{1\wedge \exp\left[ \beta \left(\frac{\eta}{n \delta } + 2m_0+ 2h + \frac{2}{n} \right)\right] \right\}[f(\eta +\delta) -f(\eta)] \nonumber 
	\\ & + n^\al \left[\frac12   [1+(m_0 + \frac{\eta}{n\delta})\right] \left\{1\wedge \exp\left[ \beta \left(-\frac{\eta}{n \delta } - 2m_0- 2h + \frac{2}{n} \right)\right] \right\}
	[f(\eta -\delta) -f(\eta)] \nonumber \\=& n^\al p_{+}^n(\eta) [f(\eta +\delta) -f(\eta)] + n^\al p_{-}^n(\eta) [f(\eta -\delta) -f(\eta)] \label{eqn:generator_n}
\end{align}
with $\al$ being the rate of the transition, i.e. the time scaling factor, a parameter that we can control, and $\delta=n^{\gamma -1}$ the space scaling factor. %moreover,  
	
To estimate \eqref{eqn:generator_n}, apply Taylor expansion, we have, for some $\chi \in [\eta, \eta+\delta]$ and $\zeta \in [\eta-\delta, \eta]$, 
\begin{align*}
	& n^\al p_{+}^n(\eta) [f(\eta +\delta) -f(\eta)] +  n^\al p_{-}^n(\eta) [f(\eta -\delta) -f(\eta)] 
	\\ =&   n^\al p_{+}^n(\eta) [f'(\eta) \delta + \frac12 f''(\eta) \delta^2] +n^\al  p_{-}^n(\eta) [-f'(\eta)\delta+\frac12 f''(\zeta)\delta^2] 
	\\ = &n^\al f'(\eta) \delta [ p_{+}^n(\eta) -  p_{-}^n(\eta)] +\frac12 n^\al  f''(x) \delta^2 [ p_{+}^n(\eta) +  p_{-}^n(\eta)] 
	\\& + \frac12 n^\al \delta^2  p_{+}^n(\eta) [f''(\chi) -f''(\eta)] +   \frac12 n^\al \delta^2  p_{-}^n(\eta) [f''(\zeta) -f''(\eta)].
\end{align*}
The estimations are provided in Corollaries \ref{cor:the_plus_calculation} and \ref{cor:the_minus_calculation}, as well as the gradient bound Lammata \ref{lem: gradient_bound} and corollary \ref{cor:boundedness_of_the_third_derivative}. 
Basically, note that at the critical temperature, $h=0$, $\beta=1$, we have, $\gamma = 1/4$, hence, 
$\delta= n^{\gamma-1}=n^{-3/4}$, and $\al = 3/2$, so, $n^\al\delta= n^{3/4}$ and  $n^\al\delta^2= 1$. In this case, $G_\infty= -\frac{2}{3}x^3f'(x)+2f''(x)$. Meanwhile, at supercritical temperature, $\beta\in [0,1)$, we have, $\gamma=1/2$, hence, 
$\delta= n^{\gamma-1}=n^{-1/2}$,  and $\al = 1$, so, $n^\al\delta= n^{1'2}$ and  $n^\al\delta^2= 1$. In this case, $G_\infty f=- (1- m_0)\beta x \left[2\beta     -\frac{ (1 +\eta^2 \beta)}{1+m_0} \right] f'(x)+2(1-m_0)f''(x)$.
Thus, we have,
\begin{align*}
 E[ (G_n f- G_\infty f)( X(\infty)]&\le E[E_1(X(\infty))+E_2(X(\infty))+E_3(X(\infty))].
\end{align*} 
Where $E_i$ refers to the $i$-th order term in the approximation, for $i=1,2,3$, and from  \ref{cor:the_plus_calculation} and \ref{cor:the_minus_calculation}, we can conclude that they are in the desired order. 
\end{proof}

\bibliographystyle{APT}
\footnotesize

\bibliography{Lu}

\end{document}